\documentclass[12pt,a4paper,oneside]{amsart}

\usepackage{amsfonts, amsmath, amssymb, amsthm, amscd}
\usepackage{hyperref}
\hypersetup{
  colorlinks   = true, 
  urlcolor     = blue, 
  linkcolor    = blue, 
  citecolor   = red 
}
\usepackage{anysize}
\usepackage[pdftex]{graphicx}

\newtheorem{theorem}{Theorem}[section]

\theoremstyle{definition}
\newtheorem{definition}[theorem]{Definition}

\theoremstyle{remark}
\newtheorem{remark}[theorem]{Remark}

\numberwithin{equation}{section}

\DeclareMathOperator{\vol}{vol}

\DeclareMathOperator{\conv}{conv}

\DeclareMathOperator{\inte}{int}

\renewcommand{\epsilon}{\varepsilon}

\newcounter{fig}
\newcommand{\f}{\refstepcounter{fig} Fig. \arabic{fig}. }

\begin{document}
\ifpdf
\DeclareGraphicsExtensions{.pdf, .jpg, .tif, .mps}
\else
\DeclareGraphicsExtensions{.eps, .jpg, .mps}
\fi
	
\title[Elementary approach to\dots]{Elementary approach to closed billiard trajectories in asymmetric normed spaces}

\author{Arseniy~Akopyan{$^\spadesuit$}}

\email{akopjan@gmail.com}

\author{Alexey~Balitskiy{$^\clubsuit$}}

\email{alexey\_m39@mail.ru}

\author{Roman~Karasev{$^\diamondsuit$}}

\email{r\_n\_karasev@mail.ru}
\urladdr{http://www.rkarasev.ru/en/}

\author{Anastasia Sharipova{$^\heartsuit$}}

\email{independsharik@yandex.ru}

\thanks{$^\spadesuit$ Supported by People Programme (Marie Curie Actions) of the European Union's Seventh Framework Programme (FP7/2007-2013) under REA grant agreement n$^\circ$[291734].}
\thanks{{$^\spadesuit$}{$^\diamondsuit$} Supported by the Dynasty foundation.}
\thanks{{$^\spadesuit$}{$^\clubsuit$}{$^\diamondsuit$} Supported by the Russian Foundation for Basic Research grant 15-31-20403 (mol\_a\_ved).}
\thanks{{$^\clubsuit$}{$^\diamondsuit$} Supported by the Russian Foundation for Basic Research grant 15-01-99563 A}

\address{{$^\spadesuit$} Institute of Science and Technology Austria (IST Austria), Am Campus~1, 3400 Klosterneuburg, Austria}
\address{{$^\clubsuit$}{$^\diamondsuit$}{$^\heartsuit$}Moscow Institute of Physics and Technology, Institutskiy per. 9, Dolgoprudny, Russia 141700}
\address{{$^\spadesuit$}{$^\clubsuit$}{$^\diamondsuit$} Institute for Information Transmission Problems RAS, Bolshoy Karetny per. 19, Moscow, Russia 127994}


\begin{abstract}
We apply the technique of K\'aroly Bezdek and Daniel Bezdek to study billiard trajectories in convex bodies, when the length is measured with a (possibly asymmetric) norm. We prove a lower bound for the length of the shortest closed billiard trajectory, related to the non-symmetric Mahler problem. With this technique we are able to give short and elementary proofs to some known results.
\end{abstract}

\maketitle

\section{Introduction}

In this paper we consider billiards in convex bodies and estimate the minimal length of a closed billiard trajectory. This kind of estimates is rather useful in different practical applications, see further references on this subject in~\cite{bb2009}. 

In~\cite{aao2012} Shiri Artstein-Avidan and Yaron Ostrover presented a unified symplectic approach to handle billiards in a convex body $K\subset V$ (here $V$ is a real vector space), whose trajectory length (and therefore the reflection rule) is given by a norm with unit ball $T^\circ$ (polar to a body $T\subset V^*$ containing the origin); throughout this paper we use the \emph{possibly non-standard} notation $\|\cdot\|_T$ for this norm with $T$ lying in the dual space. 

We emphasize that in this work the norm need not be symmetric, that is need not satisfy $\|q\| = \|-q\|$. Usually the term ``Minkowski billiard'' was used, but Minkowski norms are usually assumed to be symmetric, and we do not restrict ourselves to this particular case. The idea of~\cite{aao2012} is to interpret a billiard trajectory in $K$ with norm $\|\cdot \|_T$ as a characteristic on the boundary of the convex body $K\times T\subset V\times V^*$. The space $V\times V^*$ is the cotangent bundle of $V$ and carries a natural symplectic structure, and the surface $\partial (K\times T)$, in a sense, carries a contact structure, although some effort has to be made to handle it because it is not smooth at $\partial K\times \partial T$.

The symplectic approach was rather useful and gave certain results about the number $\xi_T(K)$, that is the minimal $\|\cdot\|_T$-length of a closed billiard trajectory in $K$. In particular, in~\cite{aao2012} this number was shown to be equal to the Hofer--Zehnder capacity $c_{HZ}(K\times T)$, and it was proved that the number $\xi_T(K)$ is monotone in $T$ and $K$ under inclusions, and satisfies a certain Brunn--Minkowski type inequality. In the next paper~\cite{aaok2013} the inequality
\begin{equation}
\label{equation:symm-estimate}
\xi_{K^\circ}(K) \ge 4
\end{equation}
for centrally symmetric convex bodies was established with rather elementary techniques and it was noticed that, assuming the Viterbo conjecture for convex bodies $X\subset \mathbb R^{2n}$
\[
\vol(X) \ge \frac{c_{HZ}(X)^n}{n!},
\] 
the estimate (\ref{equation:symm-estimate}) would imply the famous Mahler conjecture~\cite{ma1939}
\[
\vol K\cdot \vol K^\circ \ge \frac{4^n}{n!}.
\]
Mahler's conjecture is known so far in a weaker form with $\frac{\pi^n}{n!}$ on the right hand side, this is a result due to Greg Kuperberg~\cite{ku2008}. More detailed information on this conjecture is given in the blog post~\cite{tao2007} of Terence Tao and the paper~\cite{aaok2013}. For the Viterbo conjecture and its possible generalizations, we recommend the paper~\cite{apb2012} and the references therein.

In this paper we invoke a more elementary and efficient approach, developed by K\'aroly Bezdek and Daniel Bezdek in~\cite{bb2009} for the Euclidean norm. It turns out that this approach remains valid without change for possibly asymmetric norms\footnote{These ideas for the Euclidean norm in the plane first appeared in~\cite{bc1989}; it was already mentioned there that more arbitrary distances (norms) can be considered similarly.}; it allows to give elementary proofs of most results of~\cite{aao2012}, worry less about the non-smoothness issues, and establish the inequality
\[
\xi_{K^\circ}(K) \ge 2 + 2/n
\]
for possibly non-symmetric convex bodies $K$ containing the origin. The latter inequality is related to the non-symmetric Mahler conjecture, see the discussion in Section~\ref{section:mahler} below.

\subsection*{Acknowledgments.}
The authors thank Yaron Ostrover for numerous remarks and corrections and the unknown referee for a huge list of corrections that helped us improve the text.

\section{Bezdeks' approach to billiards}

Let us show how the results of~\cite{aao2012} can be approached using the elementary technique of~\cite{bb2009}. First, we consider an $n$-dimensional real vector space $V$, a convex body $K\subset V$, and define
\[
\mathcal P_m(K) = \{(q_1,\ldots, q_m) : \{q_1,\ldots, q_m\} \ \text{does not fit into}\ \alpha K+t\ \text{with}\ \alpha\in (0,1),\ t\in V \}.
\]
Observe that ``does not fit into $\alpha K+t$, with $\alpha\in (0,1)$, $t\in V$'' is equivalent to ``does not fit into the interior of $K+t$ with $t\in V$''.

\begin{center}
\includegraphics{pic/bezdeks-billiards-figures-3}	

\f An element of $\mathcal P_3(K)$.
\end{center}

Then we consider a norm on $V$ such that the unit ball $T\subset V^*$ of its dual is smooth. We denote this norm by $\|\cdot \|_T$ following~\cite{aao2012}. Note that this norm need not be reversible in what follows, that is $\|q\|_T$ need not be equal to $\|-q\|_T$. 

We define the length of the closed polygonal line
\[
\ell_T (q_1,\ldots, q_m) = \sum_{i=1}^m \|q_{i+1} - q_i\|_T,
\]
where indices are always modulo $m$. So the renovated result of~\cite{bb2009} reads:

\begin{theorem}
\label{theorem:bezdeks}
For smooth convex bodies $K\subset V$ and $T\subset V^*$, the length of the shortest closed billiard trajectory in $K$ with norm $\|\cdot \|_T$ equals
\[
\xi_T(K) = \min_{m\ge 1} \min_{P\in \mathcal P_m(K)} \ell_T(P).
\]
Moreover, the minimum is attained at $m\le n + 1$.
\end{theorem}

\begin{remark}
The right hand side of the above formula is well defined without any assumption on the smoothness of $K$ and $T$. In what follows we use it as the definition of $\xi_T(K)$ even when neither $K$ nor $T$ are smooth. It makes sense to call the minimizer in this theorem a \emph{shortest generalized billiard trajectory}, which coincides with a shortest closed billiard trajectory in the case of smooth $K$ and $T$, as we will see from the proof of Theorem~\ref{theorem:bezdeks}.

A shortest generalized billiard trajectory has the following geometrical meaning. Let $p$ be a non-smooth point of $\partial K$, we consider a trajectory $\ell$ through the point $p$ as a trajectory satisfying the reflection rule for \emph{some} normal to $K$ at $p$, that is we can take an arbitrary support hyperplane to $K$ at $p$ as if it were a tangent plane (Figure~\ref{fig:reflection rule}).

The shortest generalized billiard trajectory in an obtuse triangle is shown in~Figure~\ref{fig:obtuse triangle}. It is a well known open problem whether there is a legal (not passing through any vertex) closed billiard trajectory in every obtuse triangle. 
\end{remark}

\begin{center}
\begin{tabular}{p{7.5cm}p{7.5cm}}
	\includegraphics{pic/bezdeks-billiards-figures-23} &
	\includegraphics{pic/bezdeks-billiards-figures-24}	\\
	
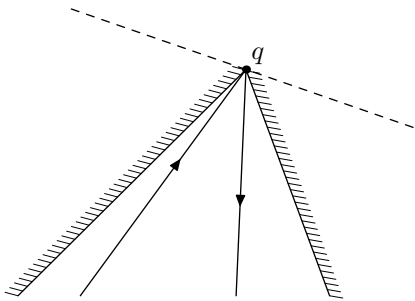
	\f \label{fig:reflection rule} The reflection rule at a non-smooth point &
	\f \label{fig:obtuse triangle} The shortest generalized billiard trajectory in an obtuse triangle.
\end{tabular}
\end{center}

\begin{proof}[Proof of Theorem~\ref{theorem:bezdeks}]
The proof in~\cite[Lemma~2.4]{bb2009} is given for the Euclidean norm; the same argument works in this more general case. We reproduce the steps here.

First, let us recall the reflection rule (see~\cite{aaok2013}, for example): For a billiard trajectory $\{q_1, \ldots, q_m\}$ we have in $V^*$
\begin{equation}
\label{equation:reflection}
p_{i+1} - p_i = - \lambda_i n_K(q_i),\quad \lambda_i > 0.
\end{equation}
This reflection rule is obtained by using the Lagrange multiplier method to optimize the expression $\|q_{i+1} - q_i\|_T + \|q_i - q_{i-1}\|_T$ varying $q_i$ under the assumption that $q_i\in \partial K$. There arise the momenta $p_i$ that are obtained from the velocities 
$$
v_i = \frac{q_i - q_{i-1}}{\|q_i - q_{i-1}\|_T}
$$
by taking the differential $p = d\|v\|_T$ (recall that the differential is in the dual space). From this definition it follows that $p_i\in \partial T$, and if we want to go back and determine the velocity $v_i$ we just take
\[
v_i = d\|p_i\|_{T^\circ},
\]
resulting in $v_i\in \partial T^\circ$. Here we need the smoothness of $T$ to define velocities knowing momenta and the smoothness of $K$ to define the normals to $K$.

The normal $n_K$ at a boundary point of the convex body $K$ is also considered as a linear functional in $V^*$ of unit norm, having maximum on $K$ precisely at this point. After summation over $i$ in (\ref{equation:reflection}) we obtain
$$
\sum_i \lambda_i n_K(q_i) = 0,
$$ 
that is the normals at the bounce points $q_i$ surround the origin in $V^*$. This means that the set $\{q_1, \ldots, q_m\}$ cannot be covered by a smaller positive homothet of $K$. Indeed, assume that a homothet $\alpha K + t$ with $\alpha\in(0,1)$ covers all the points $\{q_i\}$, therefore the translate $K+t$ of $K$ contains $q_i$'s in its interior; here we assume that the origin of $V$ is contained in $K$ without loss of generality. Let $n_i$ be the normal (linear form) such that 
\[
\max_{q\in K} \langle n_i, q\rangle = \langle n_i, q_i\rangle.
\]
By the assumption that $\inte (K+t) \ni q_i$,
\[
\langle n_i, t\rangle + \max_{q\in K} \langle n_i, q\rangle  = \max_{q\in K} \langle n_i, q+t\rangle  > \langle n_i, q_i\rangle = \max_{q\in K} \langle n_i, q\rangle,
\]
hence $\langle n_i, t\rangle > 0$, and summing such inequalities, we obtain
\begin{equation}
\label{equation:no-translate}
\left\langle \sum_i \lambda_i n_i, t\right\rangle = \langle 0, t\rangle > 0,
\end{equation}
which is a contradiction. We conclude that a shortest closed billiard trajectory $Q_{min} = \{q'_1, \ldots, q'_{m'}\}$ must be an element of some $\mathcal P_{m'}(K)$.

Now we go in the opposite direction and consider a polygonal line $Q=\{q_1,\ldots, q_m\}\in\mathcal P_m(K)$ on which the minimum is attained, including the minimum with respect to varying $m$. The previous paragraph shows that $\ell_T(Q) \le \ell_T(Q_{min})$. Applying the Helly theorem, we readily see that we can replace $Q$ by a subset with at most $m\le n+1$ points keeping the property of not fitting into a smaller homothet of $K$. In order to finish the proof, we must show that $Q$ is a generalized billiard trajectory on $K$.

We can find a translate $K+t$ that contains $Q$; such a translate must exist because otherwise we could take a smaller homothet of $Q$, still not fitting into the interior of $K$; so $Q$ would not be the length minimizer in $\mathcal P_m(K)$. By~\cite[Lemma~2.2]{bb2009}, the assumption that $Q$ does not fit into a smaller homothet of $K$ is certified, possibly after omitting the $q_i$ lying in the interior of $K+t$, by considering a set of halfspaces $H^+_i\supseteq K+t$, with respective complementary halfspaces $H^-_i$ supporting $K+t$ such that $q_i\in H^-_i\cap K$, and the intersection $\cap_{i=1}^m H^+_i$ is \emph{nearly bounded} (that is lies between two parallel hyperplanes). This actually means that the outer normals $n_i$ to $K+t$ at the $q_i$ can be non-negatively combined to zero. From here on we assume without loss of generality that $t=0$ and write $K$ instead of $K+t$.

We then observe that varying the $q_i$ inside their respective $H^-_i$ (and allowing to get outside $K$) we never obtain a configuration that can be put into a smaller homothet of $K$, because a smaller homothet of $K$ has to miss some $H^-_i$. This is established by the same argument with normals surrounding the origin resulting in (\ref{equation:no-translate}). Now let us try to minimize the length $\ell_T (q_1,\ldots, q_m)$ over 
\[
\mathcal H = \{(q_1, \ldots, q_m) : \forall i\ q_i\in H^-_i\}.
\] 

We have shown that $\mathcal H \subseteq \mathcal P_m(K)$ and therefore $Q$ is also a length minimizer in $\mathcal H$. Now we conclude from minimizing the length that every $q_i$ must either be a ``fake'' vertex where $Q$ actually does not change its direction, or a vertex where $Q$ reflects from $H^-_i$ according to (\ref{equation:reflection}); the latter is readily obtained with the Lagrange multiplier method from the minimal length assumption. The ``fake'' vertices may be again omitted keeping the property $Q\in \mathcal P_m(K)$ with $m\le n+1$, since the triangle inequality holds for asymmetric norms as usual if we keep the order of the points. The reflection points $q_i$ are on $\partial K$, and the normals to $K$ at $q_i$ must equal the normals to the respective $H^+_i$. So we conclude that $Q$ is a billiard trajectory of $K$ obeying (\ref{equation:reflection}) and $\ell_T(Q) \ge \ell_T(Q_{min})$. Since the opposite inequality is established in the first half of the proof, the proof is complete.
\end{proof}

\section{Derivation of classical and of one new result}

\subsection{Monotonicity of $\xi_T(K)$}

Let us show how the results of~\cite{aao2012} on the function $\xi_T(K)$ follow easily from Theorem~\ref{theorem:bezdeks}. First, the monotonicity 
\begin{equation}
\label{equation:monotonicity}
\xi_T(K) \le \xi_T(L)\ \text{when }\ K\subseteq L 
\end{equation}
follows easily because $\mathcal P_m(K)\supseteq \mathcal P_m(L)$ and the minimum can only get smaller on a larger set.

\subsection{Symmetry}

To prove the Brunn--Minkowski type inequality, like in~\cite{aao2012}, we need the following equality:
\begin{equation}
\label{equation:symmetry}
\xi_T(K) = \xi_K(T).
\end{equation}
This is obvious in the symplectic approach; the idea~\cite{tabachnikov2005geometry} is essentially that closed billiard trajectories correspond to critical points of the action functional 
\[
\sum_{i=1}^m \langle p_{i+1}, q_{i+1} - q_i\rangle = \sum_{i=1}^m \langle p_{i} - p_{i+1}, q_i\rangle
\]
with constraints $q_1,\ldots, q_m\in \partial K$ and $p_1,\ldots, p_m\in \partial T$, and the value of this functional at a critical point equals 
\[
\sum_{i=1}^m \|q_{i+1}- q_i\|_T = \sum_{i=1}^m \|p_{i} - p_{i+1}\|_K.
\]
This argument uses the smoothness of $K$ and $T$ in an essential way, but the monotonicity property allows to approximate any convex body by smooth bodies from below and from above, and then to pass to the limit.

\subsection{Brunn--Minkowski-type inequality}

Having noted all this, we observe that for the Minkowski sum $S+T$ in $V^*$ we have in $V$: 
\[
\|\cdot \|_{S+T} = \|\cdot \|_S + \|\cdot \|_T.
\]
Then it follows that
\[
\xi_{S+T}(K) \ge \xi_S(K) + \xi_T(K)
\]
because the minimum of the sum of functions is no less that the sum of the minima. After applying~(\ref{equation:symmetry}) this reads: 
\begin{equation}
\label{equation:bm}
\xi_T(K+L) \ge \xi_T(K) + \xi_T(L).
\end{equation}

\subsection{Estimates on $\xi_T(K)$}

We can even prove something new with this technique, or the technique of~\cite{aao2012}.

\begin{definition}
Following~\cite{bb2009}, we call $K$ \emph{$2$-periodic with respect to $T$} if one of its shortest generalized billiard trajectories bounces on $\partial K$ only twice. 
\end{definition}

We recall the main result of~\cite{aaok2013}:

\begin{theorem}[Artstein-Avidan, Karasev, Ostrover, 2013]
\label{theorem:aaok}
If $K$ and $T$ are centrally symmetric and polar to each other $(T=K^\circ)$ then $\xi_T(K) = 4$. $K$ is $2$-periodic with respect to $T$ and every segment $[-q, q]$, for any $q\in\partial K$, is a shortest generalized billiard trajectory.
\end{theorem}

\begin{remark}
There may be other minimal trajectories that are not $2$-bouncing if $K$ is not strictly convex. This can be seen already for the square $K=[-1,1]^2$.
\end{remark}

Having developed the appropriate technique, we give:

\begin{proof}[The short new proof of Theorem~\ref{theorem:aaok}]
Let us show that $\xi_T(K) \ge 4$. From Theorem~\ref{theorem:bezdeks} we conclude that it is sufficient to show that any closed polygonal line of length (in the given norm) less than $4$ can be covered by an open unit ball. This is done with the well-known folklore argument that follows.

Assume a closed polygonal line $P$ has length less than $4$. Take points $x,y\in P$ that partition $P$ into two parts of equal lengths, each part will have length less than $2$. For any $z\in P$, lying in either of the parts, we compare the straight line segments and the segments of $P$ and deduce
\[
\|z-x\| + \|z-y\| < 2
\]
from the triangle inequality.

Let $o$ be the midpoint of the segment $[xy]$. From the triangle inequality we also have
\[
\|z - o\| \le \frac{1}{2}\left(\|z-x\| + \|z-y\|\right) <1.
\]
So we have proved that $P$ is covered by an open ball (a translate of the interior of $K$) with radius $1$ centered at $o$. By Theorem~\ref{theorem:bezdeks} this is not a billiard trajectory in $K$.

So $\xi_T(K) \ge 4$ and actually the equality holds since every segment $[q,-q]$ with $q\in\partial K$ is a billiard trajectory of length $4$.
\end{proof}

\begin{center}
\includegraphics{pic/bezdeks-billiards-figures-10}	

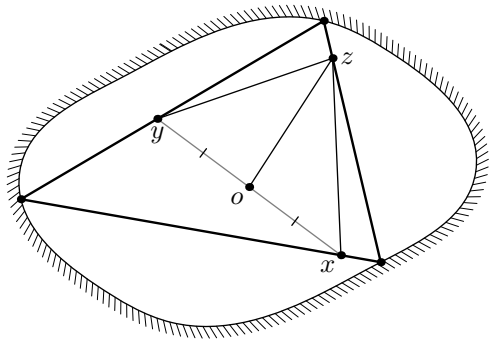
\f \label{fig:covering by ball} Explanation of the proof of Theorem~\ref{theorem:aaok}.
\end{center}

\begin{remark}
Let $K$ be strictly convex. If the length of $P$ were $4$ then in the above argument the equality $\|z-x\| + \|z-y\| = 2$ will hold at most once on either half of $P$. So a translate of $K$ covers $K$ and $P$ has at most $2$ bounces. Actually, one bounce is impossible, so a $2$-bouncing trajectory is the only case of equality, and this trajectory must be the segment $[q,-q]$ for some $q\in\partial K$. If $K$ is not strictly convex then other minimal trajectories also exist.
\end{remark}

\begin{remark}
If $K$ is a square in the plane, which is not smooth and not strictly convex, then there are plenty of minimal trajectories. Here a minimal trajectory is understood as something providing the minimum to the right hand side of the defining equation in Theorem~\ref{theorem:bezdeks}. Any segment connecting the two opposite sides of $K$ is such, and some of the quadrangles with vertices on the four sides of $K$ are also such.
\end{remark}

As another simple exercise, we establish the following result:

\begin{theorem}
\label{theorem:2bounce}
Let $K$ be $2$-periodic with respect to $T$ and let $T$ be centrally symmetric. Then $K+\lambda T^\circ$ is also $2$-periodic with respect to $T$ for any $\lambda$. 
\end{theorem}

\begin{proof}
Consider one of the shortest closed billiard trajectories in $K$ bouncing at $q_1$ and $q_2$. From Theorem~\ref{theorem:aaok} we also know that $\xi_T(T^\circ) = 4$ and we can find a pair $\{-q, q\}\in \partial T^\circ$ that gives a shortest closed billiard trajectory in $T^\circ$ with length $4$ and such that $q$ is proportional to $q_2-q_1$. The minimality assumption for $\{q_1, q_2\}$ implies that the normals $-p$ and $p$ to $K$ at $q_1$ and $q_2$ are the same as the normals to $T^\circ$ at $-q$ and $q$ respectively.

Then the pair of points $\{q_1-\lambda q, q_2+\lambda q\}$ is in the boundary of $K+\lambda T^\circ$ and the normals to $K+\lambda T^\circ$ at these points are again $-p$ and $p$. Now it follows that $\{q_1-\lambda q, q_2+\lambda q\}$ is a closed billiard trajectory in $K+\lambda T^\circ$ of length $\xi_T(K) + \lambda \xi_T(T^\circ)$. From~(\ref{equation:bm}) it follows that this trajectory is minimal.
\end{proof}

\section{Attempt toward the non-symmetric Mahler's conjecture}
\label{section:mahler}

In~\cite{aaok2013} Mahler's conjecture $\vol K\cdot \vol K^\circ \ge \frac{4^n}{n!}$ for centrally symmetric convex $n$-dimensional $K$ was reduced, assuming the Viterbo conjecture on symplectic capacities, to proving that 
\[
\xi_{K^\circ} (K) \ge 4,
\]
which is true, see Theorem~\ref{theorem:aaok} in the previous section.

Dropping the assumption of the central symmetry, the corresponding version of Mahler's conjecture becomes (see~\cite{apbtz2013}):
\[
\vol K \cdot \vol K^\circ \ge \frac{(n+1)^{n+1}}{(n!)^2}
\]
for convex bodies $K\subset\mathbb R^n$ containing the origin in the interior. Again, assuming Viterbo's conjecture, in order to deduce from it the non-symmetric Mahler conjecture, one would have to prove:
\begin{equation}
\label{equation:incorrect-bound}
\xi_{K^\circ}(K) \ge \left(\frac{(n+1)^{n+1}}{n!}\right)^{1/n},
\end{equation}
the right hand side being asymptotically $e$ by Stirling's formula. In fact, already for $n=2$ it is easy to check by hand, or look at Theorem~\ref{theorem:billiard-nonsymm} below, that the sharp estimate is 
\[
\xi_{K^\circ}(K) \ge 3,
\]
while (\ref{equation:incorrect-bound}) gives the number
\[
\left(\frac{3^3}{2}\right)^{1/2},
\]
which is greater than $3$. For higher dimensions, there also remains a gap between the actual lower bound for the billiard trajectory length and the bound needed to establish the non-symmetric Mahler conjecture, assuming the Viterbo conjecture. 

Namely, we are going to prove:

\begin{theorem}
\label{theorem:billiard-nonsymm}
If $K\subset \mathbb R^n$ is a convex body containing the origin in its interior then 
\[
\xi_{K^\circ} (K) \ge 2 + 2/n,
\]
and the bound is sharp.
\end{theorem}

This theorem shows that the non-symmetric Mahler conjecture is out of reach of the billiard approach of~\cite{aaok2013}.

\begin{proof}
We invoke Theorem~\ref{theorem:bezdeks} and consider a closed polygonal line $P$ not fitting into a smaller homothet of $K$. By the same theorem we can assume that $P$ has vertices $q_1,\ldots, q_m$ with $m\le n + 1$.

	\begin{center}
	\begin{tabular}{p{7.5cm}p{7.5cm}}
		\includegraphics{pic/bezdeks-billiards-figures-14}&
		\includegraphics{pic/bezdeks-billiards-figures-15}
		\includegraphics{pic/bezdeks-billiards-figures-16}\\
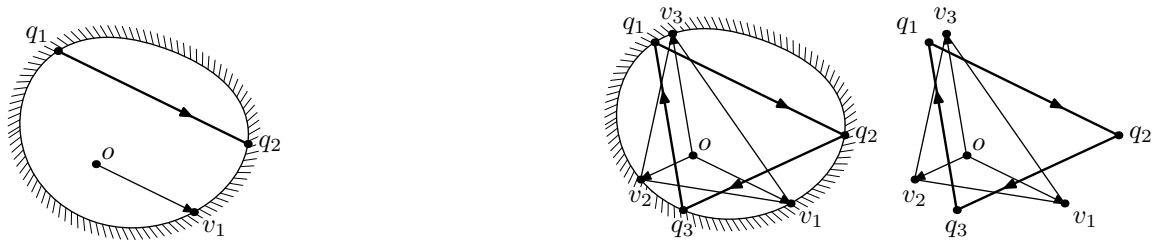
		\f \label{fig:length of the segment} Measuring the length of a directed segment.
		&
		\f \label{fig:erasing} Replacing $K$ with $L$.
	\end{tabular}
	\end{center}

Observe that the norm $\|w\|_{K^\circ}$ of a vector $w\in V$ has a very simple meaning: Let $v\in \partial K$ be the vector positively proportional to $w$ and take 
\[
\|w\|_{K^\circ} = \frac{|w|}{|v|}
\]
using the standard Euclidean norm $|\cdot|$ (Figure~\ref{fig:length of the segment}). Now to measure the length of $P$ we take $v_1,\ldots, v_m\in\partial K$ to be positively proportional to $q_2-q_1, \ldots, q_1-q_m$ respectively; it follows that the origin can be expressed as a positive combination of the vectors $\{v_i\}_{i=1}^m$. If we replace $K$ by the body $L = \conv \{v_i\}_{i=1}^m$ of possibly smaller dimension, then it is easy to see that $L$ still contains the origin and
\[
\ell_{K^\circ}(P) = \ell_{L^\circ}(P),
\] 
since $v_i$'s are still on the boundary of $L$ (Figure~\ref{fig:erasing}). Moreover, $P$ cannot fit into a smaller homothet of $L$, since it does not fit into a smaller homothet of the larger body $K\supseteq L$. In this argument $\dim L$ may become less than $\dim K$; in this case we use induction on dimension, since we have the monotonicity of the estimate $2+2/(n-1) > 2 + 2/n$. The other case $\dim K = \dim L = n$ is only possible when $m=n+1$. We can therefore assume from the start that $L$ is a simplex.

Now we are in the following situation, changing the indexing of vertices slightly. $L$ is a simplex with vertices $v_0, \ldots, v_n$ and their respective opposite facets $F_0,\ldots, F_n$, and $P$ is a closed polygonal line with vertices $q_0,\ldots, q_n$. From the first step of our construction, the following relations hold:
\begin{equation}
\label{equation:polyline}
q_{i+1}-q_i = t_i v_i, \quad t_i > 0.
\end{equation}

Also, we can assume that $q_0,\ldots, q_n$ lie on the boundary of $L$, otherwise we can translate $P$ and inflate $L$, keeping the condition that $P$ does not fit into a smaller homothet of $L$, having eventually $P\subseteq L$ (Figure~\ref{fig:homothety of simplex}).
By this the quantity $\ell_{L^\circ}(P)$ may only become smaller, and either all the vertices of $P$ will be on $\partial L$ or the dimension will drop and we use induction.

\begin{center}
\begin{tabular}{cc}
	\includegraphics{pic/bezdeks-billiards-figures-17}	&
	\includegraphics{pic/bezdeks-billiards-figures-18}\\	
	
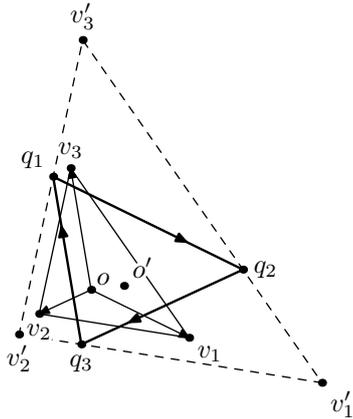
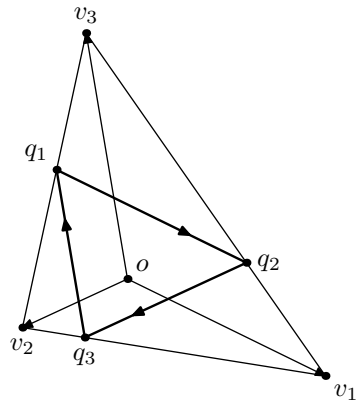
	\f \label{fig:homothety of simplex} The inflation of $L$. &
	\f \label{fig:billiard in a triangle} A billiard trajectory in the triangle.
\end{tabular}
\end{center}

So either we use induction and drop the dimension of $L$, or we use  (\ref{equation:polyline}) to conclude that the segment $[q_i, q_{i+1}]$ has direction $v_i$, the vector from the origin to a vertex of $L$. The latter implies that, if we look at $L$ along the line of sight $v_i$ then we see the facet $F_i$ and (strictly) do not see the other facets. Therefore the segment $[q_i, q_{i+1}]$ must start at $F_i$ and point into the interior of $L$, its endpoint $q_{i+1}$ must lie on some other $F_j$ ($j\neq i$), and if we extend this segment to a half-line beyond $q_{i+1}$ it must leave $L$ at $q_{i+1}$. Assuming $q_i\neq q_{i+1}$ (otherwise we have less points and the dimension drops) we obtain, in particular, that the point $q_i$ can only lie in the relative interior of its respective $F_i$.

We see that $q_i$ is the projection of $q_{i+1}$ onto $F_i$ parallel to $v_i$.
If we apply these projections cyclically starting from $q_i\in F_i$ and ending at the same point then we obtain a map that takes $F_i$ into its relative interior and that is linear on $F_i$. Such a map has a unique fixed point. So it follows that having chosen $L$ with a cyclic order on its facets we can reconstruct the considered polygonal line $P$ uniquely. 

Another way to show the uniqueness is to observe that the condition (\ref{equation:polyline}) implies $\sum_{i=0}^n t_i v_i = 0$ and therefore determines the $t_i$ uniquely up to a positive multiple. Hence the polygonal line $P$ is determined uniquely up to translation and a positive homothety, and the additional property $q_i\in F_i$ fixes it completely.

Now we are going to consider everything in barycentric coordinates. Let $(m_0,\ldots, m_n)$ be the barycentric coordinate of the origin in $L$. 
Then it is not hard to express the $q_i$ in terms of the $v_i$. We are going to index everything cyclically modulo $n+1$ and we put 
\[
M = \sum_{0\le k < l \le n} m_k m_l.
\]
From the Schur concavity of the elementary symmetric functions it follows that $M$ takes its maximum value at $m_0=\dots = m_n = \frac{1}{n+1}$ and therefore $M\le \frac{n}{2n+2}$. We have already shown the uniqueness of the $q_i$ after the choice of the order of the projections along the $v_i$ to facets. It remains to guess the expression for $q_i$ and prove that it gives the solution. Our guess is
\[
q_i=\frac{\sum\limits_{j\ne i} \sum\limits_{k=i}^{j-1} m_jm_k v_j}{M},
\]
where the inner summation goes cyclically from $i$ to $j-1$, so it is allowed that $j-1<i$. First, it is easy to observe that the sum of all coefficients in the numerator equals $M$, because every monomial $m_km_l$ is used precisely once. Therefore we have $q_i\in F_i$. Then we express the vector $q_{i+1}-q_i$:

\[
q_{i+1}-q_i=\frac{\sum\limits_{j\ne i+1} \sum\limits_{k=i+1}^{j-1} m_jm_k v_j
-\sum\limits_{j\ne i} \sum\limits_{k=i}^{j-1} m_jm_k v_j}{M} = \frac{\sum\limits_{j\ne i}m_im_j v_i-\sum\limits_{j\ne i} m_jm_i v_j}{M}.
\]

Since $\sum m_j v_j=0$, we obtain $\sum\limits_{j\ne i} m_jm_i v_j=-m_i^2 v_i$. And from $\sum_j m_j=1$ we get $\sum\limits_{j\ne i}m_im_j v_i+m_i^2 v_i=m_i v_i$.
Finally,
\[
q_{i+1}-q_i=\frac{m_i v_i}{M} \text{ and } t_i=\frac{m_i}{M}.
\]

Now we can bound the sum of $t_i$ from below:
\[
\sum\limits_i t_i=\sum\limits_{i}\frac{m_i}{M}=\frac{1}{M}\geq \frac{2n+2}{n}.
\]
This means that the length of $P$ in the norm with unit ball $L$ is at least $2 +2/n$, and with all $m_i$ equal this bound is actually attained. 

Since it is possible to approximate $L$ by a smooth body, whose polar is also smooth, keeping the trajectory and its length the same, we conclude that the bound is sharp even in the class of smooth bodies $K$ with smooth polars.
\end{proof}

\begin{remark}
A more rigorous analysis of the trajectory $q_1\dots q_{n+1}$ (Figure~\ref{fig:billiard in a triangle}) shows that a trajectory in the simplex passing through every facet is locally minimal if and only if its segments are parallel to the segments $ov_i$ in some order.

One curious thing follows from the proof of the theorem. If we fix a simplex $L$ with the origin inside then there are $(n-1)!$ cyclic orders on the vertices, and therefore $(n-1)!$ trajectories inscribed in it with edges parallel to the respective vectors $v_i$. These (billiard) trajectories are evidently different, but all corresponding edges in all the trajectories have the same length.

One consequence of this observation is that if we consider a trajectory $q_0\dots q_n$ and draw the hyperplane $h_i$ through the midpoint of every segment $q_iq_{i+1}$, parallel to the facet $F_i$ of $L$, then all these hyperplanes $h_i$ intersect in a single point.
\end{remark}

\begin{center}
\includegraphics{pic/bezdeks-billiards-figures-19}	

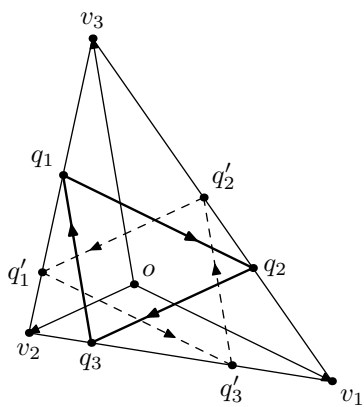
\f The two trajectories in the two-dimensional triangle.
\end{center}

The proof of Theorem~\ref{theorem:billiard-nonsymm} also reveals the following formula: Let $\ell_i$ be the length of the Cevian\footnote{\emph{Cevians} of a simplex $L$ are $n+1$ segments connecting the vertices $v_i$ with their respective opposite facets $F_i$ and all having a common point in the interior of $L$.} 
of $L$ passing through the vertex $v_i$ and the origin. Then for any closed polygonal line $P=(q_0,\ldots, q_n)$ with $q_i\in F_i$ and $q_{i+1}-q_i = t_i v_i$ with $t_i>0$ we have
\[
\sum \frac{|q_{i+1}-q_{i}|}{\ell_i}=2. 
\]
Indeed, $\frac{|v_i|}{\ell_i}=\sum \limits_{j\ne i} m_j$, since the $m_i$ are the barycentric coordinates of the origin. So we obtain
\[
 \sum \frac{|q_{i+1}-q_{i}|}{\ell_i} =\frac{\sum \limits_i m_i\left(\sum \limits_{j\ne i} m_j\right)}{M}=\frac{2M}{M}=2.
\]

\bibliography{../bib/karasev}
\bibliographystyle{abbrv}
\end{document}